\documentclass[a4paper,reqno]{amsart}

\usepackage{amssymb}
\usepackage{amsfonts}
\usepackage{mathtools}
\mathtoolsset{showonlyrefs=true}

\usepackage{hyperref}
\hypersetup{% option list for hyperref
 setpagesize=false,
 hypertexnames=false,
 hyperfootnotes=true,
 bookmarksnumbered=true,%
 bookmarksopen=true,%
 colorlinks=true,%
 linkcolor=blue,
 citecolor=blue,
}
\theoremstyle{plain}
  \newtheorem{theorem}{Theorem}[section]
  \newtheorem{lemma}[theorem]{Lemma}
  \newtheorem{proposition}[theorem]{Proposition}
  \newtheorem{corollary}[theorem]{Corollary}
\theoremstyle{definition}
  
  \newtheorem*{acknowledgement}{Acknowledgement}
  \newtheorem*{dataavailability}{Data availability}
  \newtheorem*{aideclaration}{Declaration of generative AI use}
  \newtheorem*{conflict}{Conflict of Interest}
\theoremstyle{remark}
  \newtheorem{remark}[theorem]{Remark}

\numberwithin{equation}{section}

\renewcommand{\l}{\left}
\renewcommand{\r}{\right}
\newcommand{\eps}{\varepsilon}
\newcommand{\R}{\mathbb{R}}

\newcommand{\cE}{\mathcal{E}}

\newcommand{\cV}{\mathcal{V}}
\newcommand{\Pe}{\widetilde P}
\newcommand{\1}{\mathbf{1}}

\def\norm#1{\left\Vert #1 \right\Vert}

\DeclareMathOperator{\re}{Re}
\DeclareMathOperator{\im}{Im}

\DeclareMathOperator{\diag}{diag}

\newcommand{\semicolon}{\mathrel{;}}
\title[Scattering for NLS system without mass resonance or radial symmetry]{Scattering for the quadratic NLS system in $\mathbb{R}^{5}$ without mass resonance or radial symmetry}
\author[T. Inui]{Takahisa Inui}
\address[T. Inui]{Department of Mathematics, Graduate School of Science, The University of Osaka, Toyonaka, Osaka, Japan 560-0043.}
\email{inui@math.sci.osaka-u.ac.jp}

\author[K. Nishimura]{Kuranosuke Nishimura}
\address[K. Nishimura]{ARISE Analytics Inc.,THE LINKPILLAR 1 NORTH 10F, 2-21-1 Takanawa, Minato-ku, Tokyo, Japan 108-8618.}
\email{1117614@alumni.tus.ac.jp}

\date{\today}

\begin{document}

\begin{abstract}
We consider the quadratic NLS system
\begin{equation*}
 \begin{cases}
  i\partial_tu+\Delta u=v\overline u,\\
  i\partial_tv+\kappa\Delta v=u^2
 \end{cases}
 \qquad (t,x)\in\R\times\R^5,
\end{equation*}
where $\kappa>0$. If $\kappa=1/2$, which is called mass resonance condition, then scattering below the ground state is shown by \cite{Hamano2018}. Moreover, when $\kappa \neq 1/2$, scattering of radial solutions is proved in \cite{MR4360610}. In the present paper, we prove scattering below the ground state for $\kappa\neq1/2$ without radial symmetry. 
Our proof is based on the concentration compactness and rigidity method by Kenig--Merle~\cite{MR2257393}. 
For the rigidity argument, following Pausader~\cite{MR2779061}, we use a localized virial argument in the direction orthogonal to the momentum.
\end{abstract}

\maketitle

\tableofcontents

\section{Introduction}

\subsection{Motivation}

We consider the following quadratic nonlinear Schr\"{o}dinger system on $\mathbb{R}^5$: 
\begin{equation}
\label{NLS}
\tag{NLS}
 \begin{cases}
  i\partial_tu+\Delta u=v\overline u,\\
  i\partial_tv+\kappa\Delta v=u^2,\\
  (u,v)|_{t=0}=(u_0,v_0)\in H^1(\R^5)\times H^1(\R^5).
 \end{cases}
\end{equation}
where $\kappa>0$. 
See \cite{MR3082479} for the derivation of the equation. 
Under the scaling
\[
 (u,v)(t,x)\longmapsto
 \alpha^2(u,v)(\alpha^2t,\alpha x),
 \qquad \alpha>0,
\]
the equation \eqref{NLS} is invariant and the critical Sobolev regularity is $s_c=1/2$. Hence, \eqref{NLS} is $\dot{H}^{1/2}$-critical, i.e. $L^2$-supercritical and $\dot{H}^1$-subcritical. 

It is known that the equation \eqref{NLS} is locally well-posed in $H^1(\mathbb{R}^5) \times H^1(\mathbb{R}^5)$ (see \cite{MR3082479}). Moreover, the energy $E$, the mass $M$, and the momentum $P$ are conserved, where 
\begin{align*}
	M(u,v)&\coloneqq \norm{u}_{L^2}^2+\norm{v}_{L^2}^2, \\
	E(u,v)&\coloneqq \norm{\nabla u}_{L^2}^2+\frac{\kappa}{2}\norm{\nabla v}_{L^2}^2
	+\re\int_{\R^5}\overline v u^2\,dx, \\
	P(u,v)&\coloneqq \im\int_{\R^5}
	\left(\overline u\nabla u+\frac12\overline v\nabla v\right)dx.
\end{align*}
Our interest is the global behavior of the solutions to \eqref{NLS}. From the viewpoint of the global dynamics, the ground state standing wave plays a threshold. The standing wave is of the form
\[
 (u,v)(t,x)=\left(e^{it}\phi(x),e^{2it}\psi(x)\right),
\]
where $(\phi,\psi)$ is a real-valued solution to
\begin{equation}
\label{stationary}
 \begin{cases}
  -\phi+\Delta\phi=\phi\psi,\\
  -2\psi+\kappa\Delta\psi=\phi^2
 \end{cases}
 \qquad x\in\R^5.
\end{equation}
A nontrivial least-action solution of \eqref{stationary} is called a ground state. The existence of the ground state was also proved in \cite{MR3082479}. The ground state of \eqref{stationary} is characterized by
\begin{equation*}
	S_{1}(\phi,\psi) = \inf \{S_{1}(f,g) \semicolon (f,g) \in H^1(\mathbb{R}^5)\times H^1(\mathbb{R}^5) \setminus\{(0,0)\}, K(f,g)=0\},
\end{equation*}
where $S_{1}=\frac{1}{2}E+\frac{1}{2}M$ and $K$ is the virial functional which is given by
\begin{equation*}
	K(f,g) \coloneqq \|\nabla f\|_{L^2}^2 + \frac{\kappa}{2} \|\nabla g\|_{L^2}^2 + \frac{5}{4} \re \int_{\mathbb{R}^5} \overline{g}f^2 dx. 
\end{equation*}

The condition $\kappa=1/2$ is called mass resonance condition. In this case, \eqref{NLS} is Galilean invariant. That is, $(e^{i\xi\cdot x}e^{-it|\xi|^2}u(t,x-2t\xi), e^{2i\xi\cdot x}e^{-2it|\xi|^2}v(t,x-2t\xi))$ is the solution to \eqref{NLS} when $(u,v)$ is the solution. On the other hand, \eqref{NLS} is not Galilean invariant when $\kappa\neq 1/2$. This is the main obstacle when we try to show the scattering. 
In the present paper, we study the non-mass-resonant case $\kappa\neq1/2$ and prove the scattering result below the ground state without assuming radial symmetry.

We here focus on the previous results of \eqref{NLS} for scattering. In the mass resonant case, Hamano~\cite{Hamano2018} proved
scattering in the positive potential well below the ground-state threshold. 
For arbitrary $\kappa\neq1/2$, Hamano and the authors \cite{MR4360610} obtained the corresponding
result under radially symmetric assumption. In the paper \cite{MR4360610}, the concentration compactness part has been already constructed for non-radial solutions. Zero momentum assumption, which is verified under the radial assumption, is used to show the rigidity argument. 
Meng--Xu~\cite{MR4191343} shows the same scattering result as in~\cite{Hamano2018} when $\kappa =1/2$ by using the method of Dodson--Murphy \cite{MR3692001}. 
Wang--Yang~\cite{MR4043361} shows the scattering for nonradial solutions by the method of Dodson--Murphy \cite{MR3692001} when $\kappa$ satisfies $|\kappa - 1/2|< \varepsilon$, where $\varepsilon$ depends on $E(u_0,v_0)$, which inculeds $\kappa$, and $M(u_0,v_0)$. However, since $\varepsilon$ depends on the data and also $\kappa$, the condition is implicit and thus the scattering result below the ground state for $\kappa\neq 1/2$ is not clear.

We refer to \cite{MR4026982,MR4090442,MR4781063,ACM25,NguDin26} and references therein for studies, including blow-up results, in the $L^2$-critical case and $\dot{H}^1$-critical case.

\subsection{Main result and outline of the proof}

The main theorem of the present paper is the following. 

\begin{theorem}
\label{thm:main}
Let $\kappa>0$. Let $(\phi,\psi)$ be a ground state of~\eqref{stationary}. Suppose that $(u_0,v_0)\in H^1(\R^5)\times H^1(\R^5)$ satisfies
\begin{equation}
\label{main-assumption}
 E(u_0,v_0)M(u_0,v_0)<E(\phi,\psi)M(\phi,\psi),
 \qquad K(u_0,v_0)\geq0.
\end{equation}
Then the solution is global and scatters in both time directions, that is, there exist $(u_\pm,v_\pm)\in H^1(\R^5)\times H^1(\R^5)$ such that
\[
 \norm{(u(t),v(t))-
 \l(e^{it\Delta}u_\pm,e^{i\kappa t\Delta}v_\pm\r)}_{H^1\times H^1}
 \longrightarrow0
 \qquad (t\to\pm\infty).
\]
\end{theorem}

In the mass resonance case, i.e. $\kappa=1/2$, the theorem is known by \cite{Hamano2018} as stated above. The new part is non-mass-resonant case, i.e. $\kappa\neq 1/2$. Moreover, it is worth emphasizing that radial symmetry is NOT imposed on the initial data in the theorem.

To prove Theorem~\ref{thm:main}, we use the concentration compactness and rigidity argument of Kenig--Merle \cite{MR2257393}. We refer the reader to \cite{DHR08,FXC11,AkNa13,Gue14} for this method in the case of the scalar NLS. If Theorem~\ref{thm:main} fails, the concentration compactness argument in \cite{MR4360610} gives a nonzero critical element. Its orbit is precompact in $H^1\times H^1$ modulo translations. This part does not use radial symmetry or $P=0$.

The difficulty is the rigidity argument. Let $x(t)$ be the translation parameter of the critical element. In the non-mass-resonant case, we cannot use the Galilean invariance and thus $x(t)/t \to 0$ is not clear. To clarify the behavior of $x(t)$, the authors of~\cite{MR4360610} introduced
\begin{align}
 \Pe(t)&\coloneqq \im\int_{\R^5}
 \left(\overline u\nabla u+\kappa\overline v\nabla v\right)(t,x)\,dx,
 \label{def:Ptilde}\\
 X(t)&\coloneqq \frac{2}{M(u,v)}\int_0^t\Pe(s)\,ds,
 \label{def:X}
\end{align}
and proved
\[
 \frac{x(t)-X(t)}{t}\to 0 \quad(t\to\infty).
\]
We note that the quantity $\Pe(t)$ need not be conserved if $\kappa \neq 1/2$. To exclude the critical element, the authors also introduced the following virial identity --- actually, we used its localized version --- 
\begin{align*}
	\frac{d}{dt} \im \int (x-X(t)) \cdot \left( \overline{u}\nabla u + \frac{1}{2}  \overline{v}\nabla v  \right) dx  =2K(u,v) - \frac{2}{M(u,v)} \Pe (t) \cdot P(u,v). 
\end{align*}
The sign of the second term is not clear. This is why the rigidity theorem in \cite{MR4360610} assumes $P=0$. 

In the present paper, to overcome this difficulty, we follow an idea of Pausader \cite{MR2779061}. He used a virial identity in a direction orthogonal to the momentum for the defocusing beam equation. Choose $e=P/|P|$ if $P\neq0$, and choose any unit vector $e$ if $P=0$. Let $\Pi$ be the orthogonal projection onto the perpendicular direction to $e$, which is denoted by $e^\perp$. Let us consider the projected virial identity onto $e^{\perp}$
\begin{equation*}
	\frac{d}{dt} \im \int \Pi (x-X(t)) \cdot \left( \overline{u}\nabla u + \frac{1}{2}  \overline{v}\nabla v  \right) dx 
	=2K_{\perp}(u,v) -\frac2{M(u,v)}(\Pi\Pe(t))\cdot P(u,v),
\end{equation*}
where 
\begin{equation*}
	K_{\perp} (u,v) \coloneqq \norm{\Pi \nabla u}_{L^2}^2+\frac\kappa2\norm{\Pi \nabla v}_{L^2}^2
	+\re\int_{\R^5}\overline v u^2\,dx. 
\end{equation*}
The second term vanishes because $\Pi\Pe(t)\in e^\perp$ and $P\in\operatorname{span}\{e\}$. While the second term becomes harmless, we need to show the first term $K_{\perp}$ is uniformly positive. Pausader \cite{MR2779061} considered the defocusing problem so that we do not need to take care of the sign of the functional. However, we should show the positivity since the problem is focusing. 
To overcome this, we use a mass-preserving scaling in these four variables with respect to $e^{\perp}$. Due to this scaling, we can show that $K_{\perp}$ is positive below the ground state in the positive potential well. Thus, the precompactness of the orbit of the critical element yields a uniform-in-time lower bound for $K_{\perp}$. The above virial identity, or more precisely its localized version, yields a contradiction. 

Let us remark that our method is applicable to other equations without Galilean invariance, though it has the restriction such that we can only treat the case that the spatial dimension $d$ is greater or equal to $2$ and the nonlinearity is the $L^2(\mathbb{R}^{d-1})$-critical and supercritical (equivalently $\dot{H}^{1/2}(\mathbb{R}^d)$-critical and supercritical in terms of the exponent). Note that both restrictions do not matter in our setting of the present paper. See \cite{IKNpre} as an example, in which we consider the nonlinear fourth-order Schr\"{o}dinger equation.

\textit{Notation.} For $(u,v)\in H^1(\R^5)\times H^1(\R^5)$, set
\begin{align}
	L(u,v)& \coloneqq \|\nabla u\|_{L^2}^2+\frac{\kappa}{2}\|\nabla v\|_{L^2}^2,\label{def:L}\\
	N(u,v)&\coloneqq \re\int_{\R^5}\overline v u^2\,dx. \label{def:N}
\end{align}
Here $L$ and $N$ are the kinetic and interaction parts of the energy. Thus we have $E=L+N$. 

Let $\dot{F}$ denote the time derivative of $F(t)$.

\section{Preliminaries}
\label{sec:prelim}

\subsection{Variational setting}

For $\omega>0$, we set
\begin{equation}
	S_{\omega}(u,v)\coloneqq \frac12E(u,v)+\frac\omega2M(u,v),
\label{def:I}
\end{equation}
\begin{equation}
 \mu_\omega\coloneqq \inf\left\{
 S_{\omega}(f,g)\semicolon (f,g)\in
 \left(H^1(\R^5)\times H^1(\R^5)\right)\setminus\{(0,0)\},
 \ K(f,g)=0
 \right\}.
\label{def:I-mu}
\end{equation}
For the rescaled ground state
\[
 (\phi_\omega,\psi_\omega)(x)
 \coloneqq \omega(\phi,\psi)(\sqrt\omega x),
\]
the variational theory developed in \cite{MR4360610} gives
\begin{equation}
 \mu_\omega= S_{\omega}(\phi_\omega,\psi_\omega).
\label{ground-action}
\end{equation}
The condition
\begin{equation}
 E(f,g)M(f,g)<E(\phi,\psi)M(\phi,\psi)
\label{EM-threshold}
\end{equation}
holds if and only if $ S_{\omega}(f,g)<\mu_\omega$ for some $\omega>0$.

The following invariance is proved in~\cite[Lemma~2.10]{MR4360610}.
\begin{lemma}[Invariance by the flow]
\label{lem:sign}
Let $\omega>0$, $(u_0,v_0)\in H^1(\R^5)\times H^1(\R^5)$, and 
$(u,v)$ be the solution with $(u,v)(0)=(u_0,v_0)$. Assume that
\[
  S_{\omega}(u_0,v_0)<\mu_\omega,
 \qquad K(u_0,v_0)> 0.
\]
Then $(u,v)$ is global and $K(u(t),v(t))> 0$ for all $t$.
\end{lemma}

\begin{remark}
Under the condition $S_{\omega}(u_0,v_0)<\mu_\omega$, $K(u_0,v_0)=0$ happens if and only if $(u_0,v_0)=(0,0)$. 
\end{remark}

\subsection{Critical element}

The following proposition is proved in~\cite[Propositions~1.2 and~2.17]{MR4360610}. Its proof uses neither radial symmetry nor zero momentum.

\begin{proposition}[Critical element]
\label{prop:critical}
Suppose that forward scattering in Theorem~\ref{thm:main} fails. Then there exist $\omega>0$, a nonzero global nonscattering solution $(u^*,v^*)$, and $x\in C([0,\infty);\R^5)$ such that
\begin{equation}
S_{\omega}(u^*_0,v^*_0)<\mu_\omega,
 \qquad K(u^*_0,v^*_0)>0,
\label{critical-well}
\end{equation}
and
\[
 \{(u^*,v^*)(t,\cdot+x(t))\semicolon t\in[0,\infty)\}
\]
is precompact in $H^1(\R^5)\times H^1(\R^5)$.
\end{proposition}

We also use the following two lemmas, which are shown in~\cite[Lemmas~2.26 and~2.28]{MR4360610}.

\begin{lemma}
\label{lem:tail}
Let $(u,v)$ be a global solution on $[0,\infty)$. Suppose that $\{(u,v)(t,\cdot+x(t))\semicolon  t\in[0,\infty)\}$ is precompact in $H^1(\mathbb{R}^5)\times H^1(\mathbb{R}^5)$ for some continuous $x(t)$. Then, for any $\eps>0$, there exists $R_0>0$ such that
\begin{equation}
\sup_{t\geq0}\int_{|x-x(t)|\geq R_0}
 \left(|\nabla u|^2+|u|^2+|\nabla v|^2+|v|^2+|v||u|^2\right)(t,x)dx
 \leq\eps.
\label{tail-estimate}
\end{equation}
\end{lemma}

\begin{lemma}
\label{lem:center}
Assume the hypotheses of Lemma~\ref{lem:tail}. Then it holds that
\begin{equation}
 \frac{x(t)-X(t)}{t}\to 0
 \quad(t\to\infty),
\label{center-tracking}
\end{equation}
where $X$ is defined in \eqref{def:X}. 
\end{lemma}

\section{Proof of Theorem \ref{thm:main}}
\label{sec:transverse}

\subsection{Positivity of $K_{\perp}$}
For a fixed $e\in\mathbb{S}^4$, let $\Pi=\Pi_e$ denote the orthogonal projection onto $e^{\perp}$, defined by
\begin{equation}
 \Pi \xi = \Pi_e\xi\coloneqq \xi-(\xi\cdot e)e,
 \qquad \xi\in\mathbb{R}^5,
\label{def:projection}
\end{equation}
and set
\begin{align}
 \nabla_\perp&\coloneqq \Pi\nabla,\\
 L_\perp(u,v)&\coloneqq \norm{\nabla_\perp u}_{L^2}^2+
 \frac\kappa2\norm{\nabla_\perp v}_{L^2}^2,\\
 K_\perp(u,v)&\coloneqq L_\perp(u,v)+N(u,v).
\label{def:Kperp}
\end{align}

Note that $L_\perp(u,v)>0$ if $(u,v) \neq (0,0)$.

\begin{proposition}[Positivity]
\label{prop:Kperp-positive}
Let $\omega>0$. If $(u,v)\in H^1(\R^5)\times H^1(\R^5)$ satisfies
\begin{equation}
 S_{\omega}(u,v)<\mu_\omega,
 \qquad K(u,v)\geq0,
 \qquad (u,v)\neq(0,0),
\label{positive-well}
\end{equation}
then, for every $e\in\mathbb{S}^4$, it holds that
\begin{equation}
 K_\perp(u,v)>0.
\label{Kperp-positive}
\end{equation}
\end{proposition}

\begin{proof}
By rotational invariance, we may assume that $e=(1,0,0,0,0)$. We write $x=(x_1,\widetilde{x})\in\R\times\R^4$. Consider the following $L^2$-invariant scaling with respect to four variables $\widetilde{x}$,;  
\begin{equation}
 (u^\lambda,v^\lambda)(x_1,\widetilde{x})
 \coloneqq \lambda^2(u,v)(x_1,\lambda \widetilde{x}),
 \qquad\lambda>0.
\label{transverse-scaling}
\end{equation}
Setting $L_\parallel \coloneqq L-L_\perp$, we have
\begin{align}
 S_{\omega}(u^\lambda,v^\lambda)
 &=\frac12L_\parallel+
 \frac{\lambda^2}{2}\l(L_\perp+N\r)+\frac\omega2M,\label{I-scaling}\\
 K(u^\lambda,v^\lambda)
 &=L_\parallel+\lambda^2\left(L_\perp+\frac54N\right).
\label{K-scaling}
\end{align}
Suppose that $K_\perp(u,v)\leq0$. Since $L_{\perp}>0$, we have
\[
 N\leq-L_\perp<0,
 \qquad
 L_\perp+\frac54N=K_\perp+\frac14N<0.
\]
For $\lambda\geq1$, $S_{\omega}(u^\lambda,v^\lambda)$ is nonincreasing. Also, $K(u^\lambda,v^\lambda)$ is strictly decreasing and tends to $-\infty$ as $\lambda \to \infty$. Since $K(u,v)\geq0$, there exists $\lambda_0\geq1$ such that $K(u^{\lambda_0},v^{\lambda_0})=0$. Therefore,
\[
 \mu_\omega\leq S_{\omega}(u^{\lambda_0},v^{\lambda_0})
 \leq S_{\omega}(u,v)<\mu_\omega
\]
This is a contradiction.
\end{proof}

\begin{remark}
Actually, $K_{\perp}$ characterizes the minimizing problem $\mu_{\omega}$. See Proposition \ref{prop:B.1} for the detail. 
\end{remark}

This positivity can be extended to uniform positivity if the orbit of the solution is compact as follows. 

\begin{corollary}
\label{cor:uniform-Kperp}
Let $(u,v)$ be a nonzero global solution and satisfy $S_{\omega}(u_0,v_0)<\mu_\omega$ for some $\omega>0$. Assume that $K(u(t),v(t))>0$ for all $t\geq0$.
Suppose also that the orbit translated by some $x(t)$ is precompact in
$H^1\times H^1$. Then, for every $e\in\mathbb{S}^4$, there exists $\delta>0$
such that
\begin{equation}
 K_\perp(u(t),v(t))\geq\delta
 \qquad(t\geq0).
\label{uniform-Kperp}
\end{equation}
\end{corollary}

\begin{proof}
Otherwise, there is a sequence $\{t_n\}  \subset [0,\infty)$ such that
\[
 K_\perp(u(t_n),v(t_n))\to0.
\]
By precompactness, after taking a subsequence,
\[
 (u,v)(t_n,\cdot+x(t_n))\longrightarrow(f,g)
 \quad\text{in }H^1\times H^1.
\]
The conservation laws and continuity give $S_{\omega}(f,g)<\mu_\omega$ and $K(f,g)\geq0$. Conservation of mass gives $(f,g)\neq(0,0)$. Proposition~\ref{prop:Kperp-positive} gives $K_\perp(f,g)>0$, whereas continuity and translation invariance of $K_\perp$ give $K_\perp(f,g)=0$. This is a contradiction.
\end{proof}

\subsection{Localized virial identity}
\label{sec:virial}

If $P\neq0$, set $e=P/|P|$; if $P=0$, fix any $e\in\mathbb{S}^4$. Since $P$ is conserved, $e$ and the projection $\Pi$ do not depend on time. We of course have $\Pi P=0$. 
Choose $\chi\in C_c^\infty([0,\infty))$ so that
\[
 \chi(r)=1\quad(0\leq r\leq1),
 \qquad \chi(r)=0\quad(r\geq2),
\]
and, for $R\geq1$, set
\begin{equation}
 \chi_R(r)\coloneqq \chi(r/R),
 \qquad
 a_R(z)\coloneqq \chi_R(|z|)\Pi z.
\label{def:weight}
\end{equation}
For $X(t)$ in~\eqref{def:X}, define
\begin{equation}
 V_{\perp,R}(t)
 \coloneqq 2\im\int_{\R^5}a_R(x-X(t))\cdot
 \left(\overline u\nabla u+\frac12\overline v\nabla v\right)(t,x)dx.
\label{def:virial}
\end{equation}
Then we have the following localized virial identity. 

\begin{proposition}[Localized virial identity]
\label{prop:virial}
Let $(u,v)$ be a nonzero global solution uniformly bounded in $H^1(\mathbb{R}^5)\times H^1(\mathbb{R}^5)$. Then, for any $t\geq0$, we have
\begin{equation}
 \dot{V}_{\perp,R}(t)
 =4K_\perp(u(t),v(t))+\cE_R(t).
\label{projected-virial-clean}
\end{equation}
Moreover, there exists $C_0>0$ such that, for any $R\geq1$ and $t\geq0$,
\begin{equation}
 |\cE_R(t)|\leq C_0
 \int_{|x-X(t)|\geq R}
 \left(|\nabla u|^2+|u|^2+|\nabla v|^2+|v|^2+|v||u|^2\right)(t,x)dx.
\label{virial-error}
\end{equation}
The constant $C_0$ depends only on $\chi$, $\kappa$, $M(u,v)^{-1}$, and a uniform $H^1\times H^1$ bound of the solution.
\end{proposition}

\begin{proof}
For details, see the calculation in Appendix~\ref{app:virial}. Every term in $\cE_R$ contains $\chi_R-1$,
$\chi_R'$, or $\Delta\operatorname{div}a_R$. These coefficients are uniformly
bounded and supported in $\{|x-X(t)|\geq R\}$, which gives
\eqref{virial-error}.
\end{proof}

\begin{remark}
Since $a_{\infty}(z)=\Pi z$, which comes from the formal substitution $R=\infty$, $\operatorname{div}a_{\infty}=4$. Hence the nonlinear term has coefficient $4$, and the functional is $K_\perp=L_\perp+N$, not $K=L+\frac54N$.
\end{remark}

\subsection{Rigidity}
\label{sec:rigidity}

Combining the above arguments, we obtain the following rigidity theorem. 

\begin{theorem}[Rigidity]
\label{thm:rigidity}
Let $(u,v)$ be a global solution on $[0,\infty)$. Assume that, for some $\omega>0$,
\begin{equation}
 S_{\omega}(u_0,v_0)<\mu_\omega,
 \qquad K(u_0,v_0)\geq0.
\label{rigidity-well}
\end{equation}
Assume also that there exists $x\in C([0,\infty);\R^5)$ such that the set
\begin{equation}
 \left\{(u,v)(t,\cdot+x(t))\semicolon t\geq0\right\}
\label{rigidity-compactness}
\end{equation}
is precompact in $H^1(\R^5)\times H^1(\R^5)$.
Then $(u,v)\equiv(0,0)$. 
\end{theorem}

\begin{proof}
Suppose that $(u,v)$ is nonzero. Since $(u_0,v_0)\neq(0,0)$, the
definition of $\mu_\omega$ and \eqref{rigidity-well} imply
\[
 K(u_0,v_0)>0.
\]
Hence, Lemma~\ref{lem:sign} gives $K(u(t),v(t))>0$ for all $t\geq0$.

Let $P=P(u_0,v_0)$. When $P=0$, then the rigidity has already been shown by \cite{MR4360610}. Thus, we only consider the case $P\neq 0$. Set $e=P/|P|$ and use this $e$ to define $\Pi$ and $K_\perp$.
Proposition~\ref{prop:Kperp-positive} and
Corollary~\ref{cor:uniform-Kperp} give $\delta>0$ such that
\begin{equation}
 K_\perp(u(t),v(t))\geq\delta
 \qquad(t\geq0).
\label{rigidity-delta}
\end{equation}

Let $C_0$ be as in Proposition~\ref{prop:virial}. Lemma~\ref{lem:tail} gives $R_0\geq1$ such that
\begin{equation}
 C_0\sup_{t\geq0}
 \int_{|x-x(t)|\geq R_0}
 \left(|\nabla u|^2+|u|^2+|\nabla v|^2+|v|^2+|v||u|^2\right)dx
 \leq \frac{\delta}{2}.
\label{small-tail-rigidity}
\end{equation}

For $T_1>T_0>0$, set
\begin{equation}
 R\coloneqq R_0+\sup_{t\in[T_0,T_1]}|X(t)-x(t)|.
\label{choice-R}
\end{equation}
If $t\in[T_0,T_1]$ and $|x-X(t)|\geq R$, then
\[
 |x-x(t)|\geq |x-X(t)|-|X(t)-x(t)|\geq R_0.
\]
Thus, \eqref{projected-virial-clean}, \eqref{rigidity-delta}, and \eqref{small-tail-rigidity} give
\begin{equation}
 \dot{V}_{\perp,R}(t)\geq \frac{\delta}{2}
 \qquad(t\in[T_0,T_1]).
\label{virial-growth}
\end{equation}

Since $|a_R|\lesssim R$, the Cauchy--Schwarz inequality shows  that
\begin{equation}
 |V_{\perp,R}(t)|\leq C_1R,
\label{virial-bound}
\end{equation}
where $C_1>0$ is independent of $R$ and $t$. 

Choose $\eta>0$ so that $4C_1\eta<\delta$. Lemma~\ref{lem:center} gives $T_0$ such that
\begin{equation}
 |X(t)-x(t)|\leq\eta t
 \qquad(t\geq T_0).
\label{center-eta}
\end{equation}
For any $T_1>T_0$, this gives $R\leq R_0+\eta T_1$.

Integrating~\eqref{virial-growth} over $[T_0,T_1]$ and using \eqref{virial-bound}, we obtain
\begin{equation}
 \frac{\delta}{2}(T_1-T_0)
 \leq V_{\perp,R}(T_1)-V_{\perp,R}(T_0)
 \leq2C_1(R_0+\eta T_1).
\label{final-contradiction}
\end{equation}
Therefore, it follows that
\[
 \left(\frac{\delta}{2}-2C_1\eta\right)T_1
 \leq \frac{\delta}{2} T_0+2C_1R_0.
\]
Since $\delta/2-2C_1\eta>0$, letting $T_1\to\infty$ yields a contradiction.
Hence, we obtain $(u,v)=(0,0)$.
\end{proof}

This rigidity theorem implies the main theorem immediately.

\begin{proof}[Proof of Theorem~\ref{thm:main}]
Suppose that forward scattering fails. Proposition~\ref{prop:critical} gives $\omega>0$ and a nonzero nonscattering global solution $(u^*,v^*)$ such that
\[
 S_{\omega}(u^*_0,v^*_0)<\mu_\omega,
 \qquad K(u^*_0,v^*_0)>0,
\]
and its orbit is precompact modulo translations. This solution satisfies the hypotheses of Theorem~\ref{thm:rigidity} and thus we get a contradiction. Thus forward scattering holds. Since~\eqref{NLS} is invariant under
\[
 (u,v)(t,x)\longmapsto(\overline u,\overline v)(-t,x),
\]
the same argument gives backward scattering. 
\end{proof}

\appendix

\section{Derivation of localized virial identity}
\label{app:virial}

We prove Proposition~\ref{prop:virial} directly. By rotational invariance, it suffices to work in coordinates such that
\begin{equation}
 e=(1,0,0,0,0),
 \qquad
 \Pi=\diag(0,1,1,1,1),
 \qquad \Pi P=0.
\label{app:coordinates}
\end{equation}
Set $z=x-X(t)$ and $r=|z|$. Define
\begin{equation}
 p(u,v)\coloneqq \overline u\nabla u+\frac12\overline v\nabla v,
 \qquad n(u,v)\coloneqq \re(\overline v u^2).
\label{app:Qn}
\end{equation}

\subsection{Localized weighted term with the center $X(t)$}

We write $a_R(z)=\chi_R(r)\Pi z$ and denote its $k$-component by $(a_R)_k$ for $k=1,...,5$. Then
\begin{align}
 \partial_j(a_R)_k
 &=\chi_R(r)\Pi_{kj}
 +\frac{\chi_R'(r)}r z_j(\Pi z)_k,\label{app:Da}\\
 \operatorname{div} a_R(z)
 &=\sum_{j=1}^{5}  \partial_j(a_R)_j=4\chi_R(r)+\frac{\chi_R'(r)}r|\Pi z|^2.\label{app:diva}
\end{align}
Since $\dot{X}=2\Pe/M$, it holds that
\begin{equation}
 \partial_ta_R(z)
 =-\chi_R(r)\Pi \dot{X}
 -\frac{\chi_R'(r)}r(z\cdot \dot{X})\Pi z.
\label{app:at}
\end{equation}
The untruncated part related to the first term vanishes as follows;
\[
 -2(\Pi \dot{X}(t))\cdot P
 =-\frac4M(\Pi\Pe(t))\cdot P=0.
\]
Thus, we obtain
\begin{equation}
2\im\int_{\R^5}\partial_ta_R(z)\cdot p(u,v)dx
 =\cE_{R,1}(t),
\label{app:center-main}
\end{equation}
where
\begin{align}
 \cE_{R,1}
 \coloneqq -\frac4M\Pe(t)\cdot\im\int_{\R^5}
 \bigg[& (\chi_R-1)\Pi p(u,v) +\frac{\chi_R'}r z\bigl((\Pi z)\cdot p(u,v)\bigr)
 \bigg]dx.
\label{app:center-error}
\end{align}

\subsection{Momentum density term}

For a smooth solution $(u,v)$ and each $k=1,\dots,5$, we have
\begin{align}
 \partial_t\im p_k
 ={}&-2\sum_{j=1}^5\partial_j
 \re(\partial_j\overline u\,\partial_ku)
 -\kappa\sum_{j=1}^5\partial_j
 \re(\partial_j\overline v\,\partial_kv)\notag\\
 &+\frac12\partial_k\Delta
 \left(|u|^2+\frac\kappa2|v|^2\right)
 -\frac12\partial_kn(u,v).
\label{app:local-momentum}
\end{align}
We calculate
\[
\cV_R(t) \coloneqq 2\im\int_{\R^5}a_R(z)\cdot \partial_t p(u,v)dx. 
\]
Using \eqref{app:local-momentum} for the derivative of $p$ and integration by parts, we obtain
\begin{align}
 \cV_R(t)= {}&4\sum_{j,k=1}^5\int_{\R^5}
 \partial_j(a_R)_k\re(\partial_j\overline u\,\partial_ku)dx\notag\\
 &+2\kappa\sum_{j,k=1}^5\int_{\R^5}
 \partial_j(a_R)_k\re(\partial_j\overline v\,\partial_kv)dx\notag\\
 &-\int_{\R^5}\Delta d_R(z)
 \left(|u|^2+\frac\kappa2|v|^2\right)dx
 +\int_{\R^5}d_R(z)n(u,v)dx,
\label{app:VR}
\end{align}
where  we set $d_R\coloneqq \operatorname{div}a_R$. 
Equations~\eqref{app:Da} and~\eqref{app:diva} give
\begin{equation}
 \cV_R(t)=4L_\perp(u,v)+4N(u,v)+\cE_{R,2}(t)
 =4K_\perp(u,v)+\cE_{R,2}(t).
\label{app:virial-main}
\end{equation}
Here, the term $\cE_{R,2}$ is given by
\begin{align}
 \cE_{R,2}\coloneqq {}&
 4\sum_{j,k=1}^5\int_{\R^5}
 \left((\chi_R-1)\Pi_{kj}
 +\frac{\chi_R'}r z_j(\Pi z)_k\right)
 \re(\partial_j\overline u\,\partial_ku)dx\notag\\
 &+2\kappa\sum_{j,k=1}^5\int_{\R^5}
 \left((\chi_R-1)\Pi_{kj}
 +\frac{\chi_R'}r z_j(\Pi z)_k\right)
 \re(\partial_j\overline v\,\partial_kv)dx\notag\\
 &-\int_{\R^5}\Delta d_R(z)
 \left(|u|^2+\frac\kappa2|v|^2\right)dx\notag\\
 &+\int_{\R^5}
 \left(4(\chi_R-1)+\frac{\chi_R'}r|\Pi z|^2\right)
 n(u,v)dx.
\label{app:virial-error-explicit}
\end{align}
\subsection{Virial identity}
By \eqref{app:center-main} and \eqref{app:virial-main}, we obtain \eqref{projected-virial-clean}, that is, 
\begin{align*}
	\dot{V}_{\perp,R}(t) &= 2\im\int_{\R^5} \partial_t a_R(z)\cdot p(u,v)dx  + \cV_R(t) \\
	&=4K_\perp(u,v)+\cE_{R}(t),
\end{align*}
where we set $\cE_R\coloneqq \cE_{R,1}+\cE_{R,2}$.
For fixed $\chi$, it holds that
\begin{align}
 &\left|(\chi_R-1)\Pi_{kj}
 +\frac{\chi_R'}r z_j(\Pi z)_k\right|
 \lesssim\1_{\{r\geq R\}},\label{app:weight-bound1}\\
 &\left|4(\chi_R-1)+\frac{\chi_R'}r|\Pi z|^2\right|
 \lesssim\1_{\{r\geq R\}},\label{app:weight-bound2}\\
 &|\Delta d_R(z)|
 \lesssim R^{-2}\1_{\{R\leq r\leq2R\}}.
\label{app:weight-bound3}
\end{align}
By these estimates and the uniform boundedness of $|\Pe(t)|$, the error $\cE_R = \cE_{R,1}+\cE_{R,2}$ satisfies \eqref{virial-error}.

%%%%%%%%%%%%%%%%%%%%%%%%%%%%%%%%%%%%%%%%%%%%%%%%%%%%%%%%%

%===========================================
\section{Characterization by $K_{\perp}$}
\label{app:B}
%===========================================

Our functional $K_{\perp}$ characterizes the minimizing problem $\mu_{\omega}$ as follows. 

\begin{proposition}
\label{prop:B.1}
For $\omega>0$ and direction $e \in \mathbb{S}^{4}$, we set
\begin{equation*}
	\mu_{e}(\omega) \coloneqq \inf\{S_{\omega}(f,g) \semicolon (f,g)\in H^1(\mathbb{R}^5)\times H^1(\mathbb{R}^5)\setminus\{(0,0)\}, K_{\perp} (f,g)=0\}, 
\end{equation*}
where $K_{\perp}$ is a functional related to the direction $e^{\perp}$. 
Then it holds that
\begin{equation*}
	\mu_e(\omega) = \mu_{\omega}
\end{equation*}
and, in particular, $\mu_{e}(\omega)$ is independent of $e$. 
\end{proposition}

\begin{proof}
First, we show $\mu_{\omega} \geq \mu_{e}(\omega)$. 
Take a ground state $(\phi_\omega, \psi_\omega)$ of $\mu_{\omega}$. Then $K_{\perp}(\phi_\omega, \psi_\omega) =0$ since $(\phi_\omega, \psi_\omega)$ is the solution to \eqref{stationary}. Thus, by the definition of  $\mu_{e}(\omega)$, 
\begin{equation*}
	S_{\omega}(\phi_\omega, \psi_\omega) \geq \mu_{e}(\omega)
\end{equation*}
Therefore, it holds that $\mu_{\omega} \geq \mu_{e}(\omega)$. 

Next, we show $\mu_{\omega} \leq \mu_{e}(\omega)$ based on Proposition \ref{prop:Kperp-positive}. 
By rotational invariance, we may assume that $e=(1,0,0,0,0)$. Take $(f,g)\neq (0,0)$ such that $K_{\perp}(f,g)=0$. Since $K(f^{\lambda},g^{\lambda}) \to L_{\parallel}(f,g)>0$ as $\lambda \to +0$ and $K(f^{\lambda},g^{\lambda}) \to -\infty$ as $\lambda \to \infty$, there exists $\lambda_0 \in (0,\infty)$ such that $K(f^{\lambda_0},g^{\lambda_0}) =0$. This and $K_{\perp}(f,g)=0$ imply
\begin{equation*}
	\mu_{\omega} \leq S_{\omega}(f^{\lambda_0},g^{\lambda_0}) = S_{\omega}(f,g). 
\end{equation*} 
Thus, we obtain $\mu_{\omega} \leq \mu_{e}(\omega)$. This completes the proof.
\end{proof}

%%%%%%%%%%%%%%%%%%%%%%%%%%%%%%%%%%%%%%%%%%%%%%%%%%%%%%%%%
\begin{dataavailability}
There is no available data related to this work. 
\end{dataavailability}

\begin{conflict}
The authors declare that there are no conflicts of interest. 
\end{conflict}

\begin{aideclaration}
The authors used ChatGPT in order to prepare this manuscript. After using this tool, the authors reviewed and edited the content. They take full responsibility for the content of this manuscript.
\end{aideclaration}

\begin{acknowledgement}
The first author is supported by KAKENHI Grant-in-Aid for Early-Career Scientists No. JP24K16947 and partially by KAKENHI Grant-in-Aid for Scientific Research (B) No. JP26K00612. 
\end{acknowledgement}

%%%%%%%%%%%%%%%%%%%%%%%%%%%%%%%%%%%%%%%%%%%%%%%%%%%%%%%%%

\end{document}